\documentclass[12pr]{amsart}
\usepackage{amsmath, amssymb, graphicx, epsfig, verbatim, psfrag, color, amscd}

\newcommand\HHFm{{\mathbf{HF}}^-}

\newcommand\orM{\vec{M}}
\newcommand\orN{\vec{N}}
\newcommand\orL{\vec{L}}
\newtheorem{thm}{Theorem}[section]

\newtheorem{cor}[thm]{Corollary}
\newtheorem{lem}[thm]{Lemma}
\newtheorem{prop}[thm]{Proposition}
\newtheorem{defn}[thm]{Definition}

\newtheorem{exas}[thm]{Examples}
\newtheorem{rem}[thm]{Remark}

\numberwithin{equation}{section}

\newcommand\Field{\mathbb F}

\newcommand\Dual{\mathcal D}
\newcommand\Duality\Dual

\newcommand\Tor{\mathrm{Tor}}

\newcommand\ws{\mathbf w}
\newcommand\zs{\mathbf z}

\newcommand\x{\mathbf x}

\newcommand\p{\mathbf p}
\newcommand\q{\mathbf q}
\newcommand\y{\mathbf y}

\newcommand\ModSphere{\ModFlow\left({\mathbb S}\longrightarrow 
\Sym^{g-1}(\Sigma_{1})\times \Sym^2(\Sigma_{2})\right)}
\newcommand\ModSpheres\ModSphere

\newcommand\gr{\mathrm{gr}}

\newcommand\UnparModSp{\widehat \ModSp}
\newcommand\UnparModFlow\UnparModSp
\newcommand\Mod\ModSp

\newcommand\PD{\mathrm{PD}}

\newcommand{\spinc}{\mathfrak s}

\newcommand\ModMaps{\mathcal M}
\newcommand\ModSp\ModMaps

\newcommand\Ta{{\mathbb T}_{\alpha}}
\newcommand\Tb{{\mathbb T}_{\beta}}

\newcommand\alphas{\mbox{\boldmath$\alpha$}}

\newcommand\betas{\mbox{\boldmath$\beta$}}

\newcommand\Ring{\mathbb A}

\newcommand{\rr}{\mathbf r}
\newcommand\Chain{\mathfrak A}

\newcommand\HH{\mathbb H}
\newcommand\oHH{\overline\HH}

\newcommand{\Zmod}[1]{{\mathbb Z}/{#1}{\mathbb Z}}

\newcommand{\grrr}{{\rm {gr}}}
\newcommand{\Sym}{{\mathrm {Sym}}}
\newcommand{\s}{\mathbf s}   \renewcommand{\t}{\mathbf t}

 \newcommand{\Z}{\mathbb Z} \newcommand{\N}{\mathbb
  N} \newcommand{\Q}{\mathbb Q} \newcommand{\R}{\mathbb R}

\newcommand{\HFa}{{\widehat {\rm {HF}}}}

\newcommand{\HFmi}{{\rm {HF}}^-}

\DeclareMathOperator{\SpinC}{Spin^c}

\begin{document}

\title{A spectral sequence on lattice homology}

\author{Peter Ozsv\'ath}
\address{Department of Mathematics, Princeton University\\
Princeton, NJ, 08544}
\email{petero@math.princeton.edu}

\author{Andr\'{a}s I. Stipsicz}
\address{R{\'e}nyi Institute of Mathematics\\
Budapest, Hungary and \\
Institute for Advanced Study, Princeton, NJ}
\email{stipsicz@math-inst.hu}

\author{Zolt\'an Szab\'o}
\address{Department of Mathematics, Princeton University\\
Princeton, NJ, 08544}
\email{szabo@math.princeton.edu}


\begin{abstract}
  Using the link surgery formula for Heegaard Floer homology we find a
  spectral sequence from the lattice homology of a plumbing tree to
  the Heegaard Floer homology of the corresponding 3-manifold.  
  This spectral sequence shows that for graphs with at most two
  ``bad'' vertices, the lattice homology is isomorphic to the Heegaard
  Floer homology of the underlying 3-manifold.
\end{abstract}

\maketitle

\newcommand\CFinfComb{{\mathbb{CF}}^{\infty}}
\newcommand\HFinfComb{{\mathbb{HF}}^{\infty}}
\newcommand\HFmComb{{\mathbb{HF}}^{-}}
\newcommand\HFpComb{{\mathbb{HF}}^{+}}
\newcommand\HFComb{{\mathbb{HF}}}
\newcommand\HFaComb{{\widehat {\mathbb{HF}}}}
\newcommand\CFmComb{{\mathbb{CF}}^{-}}
\newcommand\CFpComb{{\mathbb{CF}}^{+}}
\newcommand\Char{\mathrm{Char}}
\newcommand\Vertices{\mathrm{Vert}}
  \newcommand{\Psia}{\widehat {\Psi}}
  \newcommand{\Phia}{\widehat {\Phi}}
 \newcommand\CFaComb{\widehat {\mathbb{CF}}}

\section{Introduction}
Heegaard Floer homologies were introduced in 2001 by the first and
third authors as invariants of closed, oriented 3-manifolds
\cite{OSzF1, OSzF2}.  The construction of the invariants relies on a
choice of a Heegaard decomposition of the 3-manifold at hand, and then
applies Lagrangian Floer homology to a symplectic manifold (and two
Lagrangian subspaces of it) associated to the Heegaard
decomposition. The theory comes in many variants: the version $\HFmi
(Y)$ is the most powerful in 3- and 4-dimensional applications,
while the simpler $\HFa (Y)$ turnes out to be more accessible for
computation. Since the introduction of the invariants, many results
have been found towards their computability \cite{MOS, ManOzsThu, nice,
  SW}, but a convenient computational scheme in general is still
missing.  For 3-manifolds which can be presented as the boundary of a
negative definite plumbing with at most one {\em bad vertex} (in the
sense of Definition~\ref{def:type}), a relatively simple computational
algorithm was described in \cite{OSzplum}.

Motivated by the result of \cite{OSzplum}, in \cite{lattice} Andr\'as
N\'emethi introduced an algebraic object,
the \emph{lattice homology}
for plumbed 3-manifolds, which --- when considered for negative
definite plumbings --- provides a bridge between certain analytic
properties of the singularity with resolution the given plumbing, and
the differential topology of the boundary 3-manifold. Since lattice
homology extends the combinatorial approach found in \cite{OSzplum} to
more general plumbings, it can be shown that for a negative definite
plumbing tree $G$ with at most one bad vertex, the lattice homology
$\HFmComb (G)$ and the Heegaard Floer homology group $\HFmi (Y_G)$ of
the plumbed 3-manifold $Y_G$ (obtained by plumbing circle bundles over spheres
according to $G$) are isomorphic.  Indeed, N\'emethi extended the
isomorphism of \cite{OSzplum} to a larger class of plumbing graphs
which he called almost-rational \cite{lattice}. (For the definition of
these notions, see Section~\ref{sec:second}. See also \cite{s3csomok}
for related results.)  His results can be viewed as evidence for a
conjecture that, for a plumbing tree $G$, the lattice homology
$\HFmComb (G)$ is isomorphic to the Heegaard Floer homology $\HFmi
(Y_G)$ of the corresponding 3-manifold $Y_G$.  Further evidence to the
validity of this conjecture is provided by the proof of a surgery
exact triangle in lattice homology by Greene and (independently) by
N\'emethi \cite{Josh, latticetriangle}, and by the introduction of
knot lattice homology \cite{latticeknot}, cf. also \cite{Lspace}.

In the present paper we show the existence of a spectral sequence from
the lattice homology of a tree $G$ to the Heegaard Floer homology of
the corresponding plumbed 3-manifold $Y_G$.  This spectral sequence is
derived from the surgery presentation of Heegaard Floer homology from \cite{ManOzs}, compare also \cite{OSzint,
  OSzrac}.  In the statement below, the groups ${{\HFmComb} }(G)$ and
${\mathbf {HF}} ^-(Y_G)$ denote the regular lattice and Heegaard Floer
homologies after completion (with respect to the $U$ variable).  For a
definition of $\HFmComb (G)$ see Section~\ref{sec:review}.  When the
3-manifold $Y_G$ is a rational homology sphere then the completed
versions of the homologies determine the ones defined over the
polynomial ring, cf. \cite{ManOzs}; moveover, the closed four-manifold
invariants can be defined using only the completed theory.  The main
result of the paper is:

\begin{thm}\label{thm:mainss}
  Suppose that $G$ is a plumbing tree of spheres,
  and let  $Y_G$ be the corresponding $3$-manifold.  Then there is a spectral
  sequence $\{ E_i\} _{i=1}^{\infty}$ with the properties:
  \begin{itemize}
  \item The $E_2$-term of the spectral sequence is isomorphic to the lattice
    homology ${{\HFmComb} }(G)$. 
  \item The spectral sequence converges to ${\rm {\mathbf {HF}}} ^- (Y_G)$.
  \item The lattice homology $\HFmComb(G)$ naturally splits according
    to $\SpinC$ structures over $Y_G$ (see text preceding
    Definition~\ref{def:SpinCSplitting}); similarly, $\HHFm(Y_G)$
    splits according to $\SpinC$ structures. The spectral sequence
    respects these splittings.
\item If $\s\in\SpinC(Y_G)$ is a torsion $\SpinC$ structure (e.g.  if
  $Y_G$ is a rational homology sphere, this holds for any
  $\s\in\SpinC(Y_G)$), the isomorphism of the $E_2$-term with
  $\HFmComb (G)$ preserves the absolute Maslov grading.
\item If $\s\in\SpinC(Y_G)$ is a non-torsion $\SpinC$ structure, the
  isomorphism of the $E_2$-term with $\HFmComb (G)$ preserves the
  relative Maslov grading.
  \end{itemize}
\end{thm}
\begin{rem}
The $E_{\infty}$ term of the above spectral sequence (as a sequence of
modules over $\Field [[U]]$) recovers ${\rm {\mathbf {HF}}} ^- (Y_G)$
only as a vector space over $\Field$. More information about the
$\Field[[U]]$-module structure can be obtained by applying an
analogous spectral sequence over $\Field[[U]]/U^{n}$, see
Theorem~\ref{thm:UnSpecialized} below, and also the proof of
Corollary~\ref{c:type2}.
\end{rem}

As an application, we derive the following result.  (For the defintion
of type-$n$ graphs, see Definition~\ref{def:type} in
Section~\ref{sec:second}. Negative definite type-$n$ graphs include
graphs with at most $n$ bad vertices.) See~\cite[Section~8]{s3csomok}
for special cases of this result.

\begin{cor} \label{c:type2} 
If a plumbing tree $G$ is of type-2 then the lattice homology of $G$
is isomorphic to the Heegaard Floer homology of the underlying
3-manifold $Y_G$.
\end{cor}

The paper is organized as follows. In Section~\ref{sec:second} we fix
notations and describe some necessary definitions, while in
Section~\ref{sec:review} we recall the basic concepts of lattice
homology. Section~\ref{sec:ss} is devoted to the discussion of the
spectral sequence, and finally in Section~\ref{sec:appl} we prove
Corollary~\ref{c:type2}. In this proof we use the surgery exact
sequence of Greene and N\'emethi \cite{Josh, latticetriangle}. For
completeness, in an Appendix we include a proof of this result adapted
to the conventions used throughout our paper.

{\bf Acknowledgements}: PSO was supported by NSF grant number
DMS-0804121.  AS was supported by OTKA NK81203, by the ERC Grant
LDTBud and by the \emph{Lend\"ulet program}. ZSz was supported by NSF
grants number DMS-0603940, DMS-0704053, DMS-1006006.

\section{Background}
\label{sec:second}

Suppose that $\Gamma$ is a tree on the vertex set $V=\Vertices (\Gamma
) =\{ v_1, \ldots , v_n\}$, while $G$ is the same graph together with
an integer $m_v\in \Z $ (a \emph{framing}) attached to each vertex
$v$ of $\Gamma$. Let $M_G$ denote the associated incidence matrix
(with framings in the diagonal). The plumbing 4-manifold defined by
$G$ (when we plumb disk bundles over spheres according to $G$) will be
denoted by $X_G$, and its boundary 3-manifold is $Y_G$. It is not hard
to see that $M_G$ is the intersection matrix of the 4-manifold $X_G$
in the basis $\{ E_1, \ldots ,E_n \} \subset H_2 (X_G; \Z )$ where
$E_i$ corresponds to the vertex $v_i$ ($i=1, \ldots , n$). Let $d_v$
denote the number of neighbours of a vertex $v$ in the tree $G$;
this quantity is sometimes called the degree (or valency) of the
vertex $v_i$.  Although lattice homology can be defined for graphs
containing cycles, in the present work we will restrict our attention
to trees and forests (disjoint unions of trees).

\begin{defn}\label{def:type}
\begin{itemize}
\item 
Suppose that $G$ is a negative definite plumbing tree (that is, the
matrix $M_G$ is negative definite). According to \cite{Artin} there is
a class $Z=\sum _i n_i E_i$ with $n_i\geq 0$ integers and $Z\neq 0$
which satisfies $Z\cdot E_i\leq 0$ for all $i$, and for any other
class $Z'=\sum _i n_i'E_i$ with these properties $n_i\leq n_i'$ holds
for all $i$.  The plumbing tree $G$ is called \emph{rational} if for
$Z=\sum _i n_iE_i$ we have
\[
(\sum _i n_i E_i)^2=2\sum _i n_i+\sum _i n_iE_i^2 -2.
\]
(This condition is equivalent to requiring that the geometric genus
$p(Z)=\frac{1}{2}(Z^2+K\cdot Z)+1$ of the class $Z$ vanishes.)
\item The vertex $v$ is a \emph{bad vertex} of $G$ if $d_v+m_v>0$,
  i.e. the valency of the vertex is more than the negative of its
  framing.
\item The plumbing tree $G$ is of \emph{type-$k$} if it has $k$ vertices $\{
  v_{i_1}, \ldots ,v_{i_k}\}$ on which we can change the framings $\{
  m_{i_1}, \ldots , m_{i_k}\}$ in such a way that the result is
  rational.
\end{itemize}
\end{defn}

\begin{rem}
 The above definition differs from the definition of
 N\'emethi~\cite{nemethi-ar}: we use the term {\em bad vertices} as it
 was used in \cite{OSzplum}. For negative definite trees, the notion
 of almost-rational coincides with type-1. If a negative definite tree
 $G$ has $k$ bad vertices then it is of type-$k$. The converse is
 false, cf. the example of Figure~\ref{fig:pelda}.
\end{rem}
\begin{figure}[ht]
\begin{center}
\epsfig{file=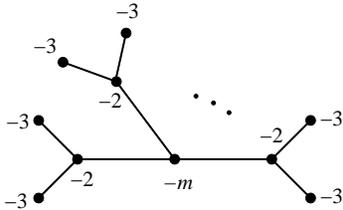, height=2.8cm}
\end{center}
\caption{{\bf The plumbing diagram of the figure has at least $n$ bad
    vertices (where $n$ is the valency of the central $(-m)$-framed
    vertex) and it is either type-1 or rational (depending on the
    actual value of $m$).}  Notice that all the $(-2)$-framed vertices
  are bad vertices (in the sense of Definition~\ref{def:type}). On the
  other hand, for $m$ sufficiently negative the graph is rational: let
  $K$ be the vector which is the sum of all $-3$-spheres, twice the
  $-2$ spheres, and once the central vertex. It follows that for any
  value of $m$ the graph is of type-1.}
\label{fig:pelda}
\end{figure}

Recall that a plumbing tree also provides a surgery diagram for the
3-manifold it represents: replace each vertex of the diagram with an
unknot, and arrange them so that two unknots link if and only if the
corresponding vertices are connected by an edge. The framings of the
unknots are given by the integers attached to the vertices of the
plumbing graph. Notice that (viewing the resulting framed link
$L=(L_1, \ldots , L_{\ell })$ as a Kirby diagram) this procedure actually
gives the 4-manifold $X_G$ with the given 3-dimensional boundary
$Y_G$. In addition, if $L', L''$ are two sublinks of the resulting link
$L$ in such a way that $L'\subset L''$, then this surgery theoretic
approach also provides a cobordism associated to the pair: attach the
4-dimensional 2-handles to $Y_{L'}$ along the components of $L''-L'$
with the framings specified by $G$.

For later reference, let $\sigma (G)$ denote the signature of the
intersection matrix $M_G$ (or equivalently, the 4-manifold $X_G$), and
define $\chi (G)$ as the cardinality $\vert V \vert $ of its vertex
set. Notice that since $G$ is a tree, $\chi (G)$ is equal to the
Euler characteristic of $X_G$ minus 1.

\section{Review of lattice homology}
\label{sec:review}
For the sake of completeness we review the basic notions of lattice
homology.  This notion was introduced by N\'emethi \cite{lattice} (see
also \cite{properties-lattice, s3csomok}).  The current presentation
is similar to the one discussed in \cite{latticeknot}, with the
difference that now we consider the completed version of the theory,
cf. Remark~\ref{rem:conv}.  Let $G$ be a given plumbing
tree/forest. Recall that $G$ is  specified by a graph $\Gamma$, together with a
map $m$ from the vertices $\Vertices(G)$ to $\Z$, and  the integer
$m(v)=m_v$ is called the {\em framing} of $v$.

Next we recall the definition of the completed version of the lattice
homology group of $G$.  The group $\HFmComb (G)$ is computed as the
homology of the combinatorial chain complex $\CFmComb(G)$, which is a
module over the ring $\Field[[U]]$ of formal power series (where
$\Field\cong \Zmod{2}$). To define it, let $\Char(G)\subset H^2 (X_G;
\Z )$ denote the set of characteristic cohomology classes on the
4-manifold $X_G$; i.e., it is the subset of those $K\in H^2(X_G)$
which have the property that
$$K\cdot c \equiv c\cdot c$$
for all $c\in H_2(X_G;\Z)$.
Let ${\mathbb {P}}(V)$ be the power set of $V=\Vertices (G)$,
so that $E\in {\mathbb {P}}(V)$ simply means that $E\subset V$.
Now, the $\Field[[U]]$-module  underlying $\CFmComb(G)$ 
is the  direct product
\begin{equation}\label{eq:prod} 
{{\CFmComb}} (G) = \Pi _{[K,E]\in \Char (G)\times
  {\mathbb {P}}(V)}\Field [[U]]\langle [K,E]\rangle.
\end{equation}

$\CFmComb(G)$ naturally admits an integral grading, called the
\emph{$\delta$-grading}.  The $\delta$-grading of an element $U^i
\otimes [K,E]$ is given by the cardinality $\vert E\vert $ of the
elements in $E$.  This grading naturally descends to a
$\Zmod{2}$-grading by considering only the parity of $\vert E \vert$.

We define the boundary map $\partial \colon \CFmComb (G)\to \CFmComb
(G)$ as follows.  Given a subset $I\subset E$, we define the
$G$-weight $f([K,I])\in \Z $ of the pair $[K,I]$ by the formula
\begin{equation}\label{eq:gweight}
2f([K,I])= \left(\sum_{v\in I}
  K(v)\right) + \left(\sum_{v\in I} v\right) \cdot \left(\sum_{v\in I}
  v\right).
\end{equation}
Moreover, for a pair $[K,E]$, we define the \emph{minimal $G$-weight}
$g([K,E])$ by the formula $g([K,E])=\min \{ f([K,I]) \mid {I\subset E}
\}$.  Next, for the vertex $v\in E\subset V$ consider the quantities
\[
A_v([K,E])=g([K,E-v]) 
\]
and 
\[
B_v([K,E])=\min \{ f([K,I]) \mid {v\in I\subset E} \}
=\left(\frac{K(v)+v\cdot v}{2}\right) + g([K+2v^*,E-v]),
\]
where $v^*$ denotes the Poincar\'e dual of the vertex $v$ (when $v$ is
regarded as an element of the second homology $H_2 (X_G, Y_G; \Z )$ of
the plumbing 4-manifold).  It follows trivially from the definition
that $\min \{ A_v([K,E]), B_v([K,E])\} =g([K,E])$. Let
\begin{eqnarray*}     a_v[K,E]=A_v([K,E])-g([K,E]) 
&{\text{and}} & b_v[K,E]=B_v([K,E])-g([K,E]).           
\end{eqnarray*}
(Now we have that $\min \{ a_v[K,E] , b_v[K,E] \} =0$.)  We define the
boundary map on $\CFmComb (G)$ by the formula
\begin{equation}\label{eq:boundary}
{\partial [K,E]} = \sum_{v\in E} U^{a_v[K,E]}\otimes [K,E-v] + \sum_{v\in E}
U^{b_v[K,E]}\otimes [K+2v^*,E-v]
\end{equation}
on $[K,E]$ and extend it to $\CFmComb (G)$ $U$-equivariantly and linearly.  
It is obvious that the boundary map drops the $\delta$-grading by one. A simple
calculation (cf. \cite{latticeknot}) shows that
\begin{lem}
The pair $(\CFmComb (G), \partial ) $ is a chain complex, that is,
$\partial ^2=0$. \qed
\end{lem}

\begin{defn}
  The homology $H_*(\CFmComb (G), \partial )$ of the chain complex
  $(\CFmComb (G), \partial )$ is the \emph{lattice homology} $\HFmComb
  (G)$ of the plumbing graph $G$.
\end{defn}

Lattice homology is the homology of an infinite direct
product. Nonetheless, it enjoys the following finiteness property:

\begin{prop}
  \label{prop:FGlattice}
  The lattice homology group $\HFmComb(G)$ is a finitely generated
  $\Field[[U]]$ module.
\end{prop}

\begin{proof}
  This can be easily seen by induction on the number of bad vertices
  in $G$, using the long exact sequence in lattice
  homology (\cite{Josh,latticetriangle}, see also
  Corollary~\ref{cor:LongExactLattice}), and using the
  result for graphs with no bad vertices as in~\cite{lattice}.
\end{proof}

\begin{rem} 
  A number of further variants can be introduced along the same lines:
  using the coefficient ring $\Field [U^{-1}, U]]$ (the field of
    fractions for the ring of formal power series in $U$)
    we get
      $\CFinfComb (G)$ and the corresponding homology theory
      $\HFinfComb (G)$. Notice that $\CFmComb (G)$ is a subcomplex of
      $\CFinfComb (G)$, hence we can consider the quotient complex
      $\CFpComb (G)$, whose homology is $\HFpComb (G)$. Setting $U=0$ in
      $\CFmComb (G)$ we get the homology theory $\HFaComb (G)$, over
      the base ring $\Field$.  More
      generally, by setting $U^n=0$ ($n\in \N$) we get the version
      $\HFaComb ^{[n]}(G)$.
\end{rem}

\begin{rem}\label{rem:conv}
  The conventional definition of lattice homology considers direct sum
  as opposed to direct product in the definition of $\CFmComb (G)$
  given in \eqref{eq:prod}. Also, the usual coefficient ring is the
  polynomial ring $\Field [U]$ rather than $\Field [[U]]$.  With the
  changes in the present definition, in fact, we consider a completed
  version of the theory.  If $G$ is negative definite, then the usual
  definition (given for example, in \cite{lattice, latticeknot}) and
  the one given above determine each other. This principle is not true
  in general, cf. the second example in \ref{ex:peldak}.  We found the
  description adapted in this paper to be in accord with the
  corresponding Heegaard Floer homology theories.
\end{rem}

The relation
\[
K\sim K' \qquad {\mbox {if and only if }} \qquad K-K'\in 2H^2(X_G,
Y_G; \Z )
\]
splits the generators into equivalence classes: $U^i\otimes [K,E]$ and
$U^j\otimes [K',E']$ are equivalent if $K\sim K'$.  This relation then
splits the the chain complex $\CFmComb (G)$ as well, and the
definition of the boundary map in \eqref{eq:boundary} shows that the
boundary map respects this splitting. Since $G$ is a tree, the
4-manifold $X_G$ is simply connected, an hence an element of $K\in
\Char (G)$ specifies a $\SpinC$ structure $\t _K$ on $X_G$, therefore
(by restricting $\t _K$ to the boundary $Y_G$) induces a $\SpinC$
structure $\s _K$ on $Y_G$. It is not hard to see that $K\sim K'$
holds if and only if $\s _K$ and $\s _{K'}$ are isomorphic $\SpinC$
structures on $Y_G$. Hence both the chain complexes and the homologies
defined above split according to the $\SpinC$ structures of $Y_G$.
Recall that $\CFmComb (G)$ admits a $\delta$-grading (given for the
generator $[K,E]$ by $\vert E \vert$), splitting the homologies
further:
\begin{defn}
  \label{def:SpinCSplitting}
  For $i\geq 0$ define $\HFmComb _i (G, \s )$ as the subgroup of
  $\HFmComb (G)$ spanned by those pairs $[K,E]$ for which $\s _K=\s$
  and $\vert E\vert =i$.
\end{defn}

Lattice homology has a further grading, the {\em Maslov grading}. This
structure is simplest to describe in the case where the underlying
$\SpinC$ structure is torsion (i.e. the first Chern class of that
$\SpinC$ structure is a torsion cohomology class). We give the grading
in that case first.

Suppose that the $\SpinC$ structure $\s _K$ associated to a generator
$U^i\otimes [K,E]$ is torsion.  In this case define the \emph{Maslov
  grading} $\grrr (U^i \otimes [K,E])$ of a generator $U^i\otimes
[K,E]$ of $\CFmComb (G)$ as
\begin{equation}\label{eq:maslov}
\grrr (U^i\otimes [K,E])=-2i +2g(K,E)+\vert E\vert +\frac{1}{4}(K^2
-3\sigma (G)-2\chi (G)).
\end{equation}
(Recall that $K^2$ is defined as the square of $nK$ divided by $n^2$,
where $nK\in H^2 (X_G, Y_G; \Z)$ and therefore it admits a cup
square. As a result we expect $\grrr (U^i\otimes [K,E])$ to be a rational
number rather than an integer.)

\begin{lem} 
\label{lem:DiffDropsOne}
(cf. \cite{latticeknot})
The boundary map drops the Maslov grading $\grrr$ by one. 
\end{lem}
\begin{proof}
  Proceed separately for the two types of components of the boundary
  map.  After obvious simplifications, according to the definition of
  $a_v[K,E]$ we have that
\[
\grrr (U^i\otimes [K,E])-\grrr (U^i \cdot U^{a_v[K,e]}\otimes [K,E-v])=
\]
\[
=2g([K,E]) +\vert E\vert + 2a_v[K,E]-2g([K,E-v])-\vert E-v\vert =1.
\]
Similarly, 
\[
\grrr (U^i\otimes [K,E])-\grrr (U^i \cdot U^{b_v[K,e]} \otimes
[K+2v^*,E-v])=1
\]
follows from the same simplifications and the definition of
$B_v([K,E])$.  
\end{proof}

We will find it convenient to use the following terminology:

\begin{defn}
  \label{def:MaslovGraded}
   A {\em Maslov graded chain complex} is a $\Q$-graded chain complex 
   over $\Field[[U]]$ with the property that 
   \begin{itemize}
     \item the differential drops grading by one and
     \item multiplication by $U$ drops grading by two.
     \end{itemize}
\end{defn}

Lemma~\ref{lem:DiffDropsOne} and Equation~\eqref{eq:maslov} together
say that for a torsion $\SpinC$ srtucture $\s$ the grading $\gr$ gives
$\CFmComb(G, \s )$ a Maslov grading, in the sense of
Definition~\ref{def:MaslovGraded}.

\begin{lem}
Suppose that $\s _K=\s _{K'}$ is a torsion $\SpinC$ structure. 
Then the difference 
\[
\grrr (U^i\otimes [K,E])-\grrr (U^j\otimes [K',E'])
\]
is an integer, and it is congruent $\mod 2$ to the difference $\vert
E\vert -\vert E'\vert$.
\end{lem}
\begin{proof}
  In the difference the terms coming from $\sigma (G)$ and $\chi (G)$
  cancel, and the ones originating from the $U$-exponents or from the
  $g$-function are obviously even. We claim that the difference
  $\frac{1}{4}(K^2-(K')^2)$ is also even.
  Since $\s _K =\s _{K'}$, we have that $K'=K+2x$ for some vector
  $x\in H^2 (X_G, Y_G; \Z )$, therefore
\[
\frac{1}{4}(K^2-(K')^2)=x\cdot (K+x),
\]
which is even since $K$ is characteristic. (Note that since $x$ is in
the relative cohomology, the above product always makes sense.)  The
only remaing terms are $\vert E\vert -\vert E'\vert$, verifying the
statement.
\end{proof}

We turn now to the non-torsion case.  In this case the term $K^2$ is
not defined, since $nK$ is not in $H^2 (X_G, Y_G; \Z )$ for any
non-zero $n$. Nevertheless, if $\s _K = \s _{K'}$, we can still
consider the difference $K^2-(K')^2$ by writing it as $(K-K')\cdot
(K+K')$. The assumption $\s_K=\s _{K'}$ then ensures that $K-K'$
admits a lift from $H^2(X_G ; \Z )$ to $H^2 (X_G, Y_G; \Z )$, hence
the above product makes sense.  This provides a possibility of
defining a relative Maslov grading.  Notice, however, that the lift of
$K-K'$ is not unique in general: by the long exact sequence of the
pair $(X_G, Y_G)$ the ambiguity for choosing such a lift lies in the
group $H^1 (Y_G; \Z )\cong H_2 (Y_G ;\Z )$. Suppose that $x$ is a lift
of $\frac{1}{2}(K-K')$ and $y\in H_2 (Y_G; \Z )$. Then the difference
we get for $K^2-(K')^2$ by using $x$ or $x+y$ can be easily computed
to be equal to $K\vert _{Y_G}(y)$. (If the restriction $K\vert _{Y_G}$
is torsion, then this evaulation is obviously zero, and we are in the
previous situation of having absolute Maslov gradings in torsion
$\SpinC$ structures.) Therefore if $d$ denotes the divisibility of
$K\vert _{Y_G}$ (that is, this cohomology class equals $d$-times a
primitive one), then the value $K^2-(K')^2$ is well-defined up to
$4d$, hence the relative Maslov grading is well-defined modulo $d$
only. (Notice that for a characteristic cohomology class $K$ the
divisibility $d$ of the restriction $K\vert _{Y_G}$ is always even.)
In summary, we have:

\begin{lem}
  Fix two generators $U^i\otimes [K,E]$ and $U^j\otimes [K',E']$
  and suppose that ${\mathbf s}_K={\mathbf s}_{K'}$ is a non-torsion $\SpinC$ structure over $Y_G$.
  Then, the {\em relative Maslov grading}
  $$\gr(U^i \otimes [K,E], U^j\otimes
  [K',E'])=-2(i-j)+2g(K,E)-2g(K,E')
  +|E|-|E'|+\frac{1}{4}(K^2-(K')^2),$$ gives a well-defined element of
  ${\mathbb Q}/{d}$, where $d$ denotes the divisibility of
  $c_1({\mathbf s}_K)$.  \qed
\end{lem}

The proof of Lemma~\ref{lem:DiffDropsOne} readily adapts to the
non-torsion case: in this case, the lattice complex is a relatively
$\Q/d$-graded Maslov-graded complex.

\begin{exas} \label{ex:peldak}
\begin{itemize}
\item Consider the example of the graph $G$ with a single vertex $v$,
  no edges and the decoration of the single vertex to be equal to
  $+1$.  Then a characteristic cohomology class $K$ can be identified
  with the odd number $K(v)$ it takes as a value on $v$.  The
  generators of $\CFmComb (G)$ are then $[2n+1, \{ v\} ]$ and
  $[2n+1]$. The boundary of $[2n+1]$ is 0, while
 \[ \partial [ 2n+1, \{ v\}]
  = \left\{\begin{array}{ll}
      [2n+1]  + U^{n+1} \otimes [2n+3] & {\text{if $n\geq -1$}} \\
      U^{-(n+1)}\otimes [2n+1]+[2n+3] & {\text{if $n< -1$}}. 
      \end{array}
    \right.\] The map $\partial$ is then obviously injective on the
    subspace given by the finite sums of elements of the form
    $[2n+1,v]$. By allowing infinite sums (as we did), the element
\[
\sum _{n=-\infty } ^{-1} U^{\frac{1}{2}(n+1)(n+2)}[2n+1,v] + 
\sum _{n=0}^{\infty}U^{\frac{1}{2}n(n+1)}[2n+1, v]
\]
generates ${{\HFmComb }}(G)$ over $\Field [[U]]$.  This shows that in
this case $\HFmComb (G)=\HFmComb _1 (G)=\Field [[U]]$.  A simple
calculation shows that this element has zero Maslov grading, in
accordance with the Heegaard Floer homological computation for the
plumbing manifold $Y_G$ given by $G$, which is diffeomorphic to $S^3$.

\item In the next example we assume that $G$ still has a single vertex
  $v$ (and no edges) and the framing of the single vertex is zero.
The underlying 3-manifold is now $S^1\times S^2$.
  The generators are of the form $[2n]$ and $[2n, v]$,
and two generators are in the same $\SpinC$ structure if and only if
the characteristic cohomology classes coincide. As always, 
 $\partial [2n]=0$. A simple calculation shows that 
\[
\partial [2n, v]=(1+U^n)[2n].
\]
Considering the theory over $\Field [U]$ (and allowing only finite
sums) the homology for the $\SpinC$ structure $n\neq 0$ is $\Field
[U]/(U^n)$, while for $n=0$ it is $\Field [U]\oplus \Field [U]$.
Working with the completed groups (and hence using the coefficient
ring $\Field [[U]]$), the term $(1+U^n)$ is invertible for $n\neq 0$
(and vanishes if $n=0$), hence according to the definition we adopted
in the present paper we have that $\HFmComb (G, \s _n)=0$ if $n\neq 0$
and $\HFmComb (G, \s _0)=\Field [[U]]\oplus \Field [[U]]$.  (A simple
computation shows that the Maslov gradings of the two generators are
$\frac{1}{2}$ and $-\frac{1}{2}$.)  Moreover, $\HFmComb _0(G, \s
_0)\cong \Field [[U]]$ and $\HFmComb _1(G, \s _0)\cong \Field [[U]]$.

This simple computation shows that for non-torsion $\SpinC$ structures
the completed theory (over the ring $\Field [[U]]$) loses some
information.  On the other hand, for torsion $\SpinC$ structures the
completed theory determines the one referred to in
Remark~\ref{rem:conv} (which is defined over $\Field [U]$).  We just
note here that the resulting homologies are again isomorphic to the
corresponding completed Heegaard Floer homology groups.
\end{itemize}
\end{exas}

\section{The spectral sequence}
\label{sec:ss}

Before turning to the proof of our main result, we need to recall some
definitions and constructions from \cite{ManOzs} (cf. also
\cite{OSzlinks}). Recall that the plumbing graph determines a link 
$L=(L_1, \ldots , L_\ell)$
in $S^3$: each vertex of the plumbing tree gives rise to an unknot and 
these unknots are linked if and only if the corresponding vertices
are connected in the graph by an edge.

\subsection{Constructions from link Floer homology}

Let $H$ denote the homology group $H_1(S^3-L; \Z )$. By fixing an
orientation on the component $L_i$, it gives rise to an oriented
meridian $\mu _i$, and these meridians generate $H$. Using these
meridians we can identify the group ring
$\Z [H]$ with the ring of Laurent polynomials on $\ell$ variables.
Define $\HH(L)$ as
\[
\{ \sum a_i \cdot [\mu _i]\mid a_i\in \Q, \quad {\mbox {and}} \quad
2a_i +\ell k (L_i, L-L_i)\in 2\Z \},
\]
where $\ell k (L_i, L-L_i)$ is the linking number of the component $L_i$
with the rest of the link. As it was discussed in \cite{ManOzs, OSzlinks},
the set $\HH(L)$ parametrizes the relative
$\SpinC$ structures on $S^3-L$.
 
Fix a multi-pointed Heegaard diagram ${\mathcal
  H}=(\Sigma,\alphas,\betas,\ws,\zs)$ representing the link $L$, as
in~\cite{OSzlinks}.  In this diagram $\ws=(w_1,\dots,w_{\ell})$ and
$\zs=(z_1,\dots,z_{\ell})$ are basepoints with the property that the
pair $w_i$ and $z_i$ represents the $i^{th}$ component of $L$. Recall
that the multi-diagram in fact specifies an orientation on the
link. When we wish to underscore this structure, we write an oriented
link as $\orL$.

Given the Heegaard diagram and a choice of $\s\in \HH(L)$, we define
the chain complex $\Chain^-({\mathcal H},\s)$ as follows.  Any
intersection point $\x \in \Ta \cap \Tb$ has a Maslov grading $M(\x)
\in \Z$ (since the link $L$ is in $S^3$) and an Alexander
multi-grading $A(\x)\in\HH(L)$, defined using the Heegaard
diagram. This Alexander multi-grading is specified (up to an overall
additive constant, i.e. by a vector), as follows. If $w_i$ and $z_i$
are the pair of basepoints belonging to the $i^{th}$ component of the
link, and $\phi\in\pi_2(\x,\y)$ is any homotopy class connecting $\x$
and $\y$, then the $i^{th}$ component $A_i(\x)$ of $A(\x)$ satisfies
\[
A_i(\x) - A_i(\y) = n_{z_i}(\phi)-n_{w_i}(\phi).
\]
In an integral homology sphere (and specifically in $S^3$) such $\phi$
always exists and the difference above is independent of the choice of
$\phi$.

Given $\s=(s_1,\dots,s_\ell)$ and $\phi$, we define the $\s$-modified
multiplicity of $\phi\in\pi_2(\x,\y)$ by the formulas:
\begin {eqnarray}
\label {eiz}
E^i_{\s}(\phi) &=&  \max \{ s_i-A_i(\x), 0\} - \max \{ s_i-A_i(\y), 0\} + 
n_{z_i}(\phi) \\
\label {eiw}
&=& \max \{ A_i(\x) -s_i, 0\} - \max\{ A_i(\y) -s_i,0\} + n_{w_i}(\phi).
\end {eqnarray}
This quantity has the following two properties:
\begin{itemize}
  \item $E^i_s(\phi)\geq 0$ if all the local multiplicities of $\phi$ are non-negative.
  \item
    If $\phi_1\in\pi_2(\p,\q)$ and $\phi_2\in\pi_2(\q,\rr)$, then for 
$\phi _1 * \phi _2 \in \pi _2 (\p, \rr )$ 
\[
E^i_s(\phi_1*\phi_2)=E^i_s(\phi_1)+E^i_s(\phi_2).
\]
\end{itemize}

Given $\s = (s_1, \dots, s_\ell) \in \HH(L)$, we define the
corresponding chain complex $\Chain^-({\mathcal H}, \s)$, which is a
free module over the algebra $\Ring = \Field[[U_1, \dots, U_\ell]]$
generated by $\Ta \cap \Tb$, and equipped with the differential:
\begin{equation}
  \label{def:DefD}
  \partial \x = \sum_{\y \in \Ta \cap \Tb} \sum_{\substack{\phi \in \pi_2(\x, \y)\\ \mu(\phi)=1 }} \# \left(\frac{{\mathcal M}(\phi)}{\R}\right) \cdot U_{1}^{E^1_{s_1}(\phi)} \cdots  U_{\ell}^{E^\ell_{s_\ell}(\phi)} \cdot \y. 
\end{equation}
Note that this complex also depends on the choice of a suitable almost
complex structure on the symmetric product. We suppress this almost
complex structure from the notation for simplicity.

According to~\cite{ManOzs}, the above complex is related to the
Heegaard Floer homology of the 3-manifold obtained as sufficiently
large surgeries on a link.  (See also~\cite{OSzint, RasmussenThesis}
for the analogues for knots.) More formally, let $\Lambda = (\Lambda
_1, \ldots , \Lambda _{\ell})\in \Z ^{\ell }$ be a vector of framings,
and let $Y_{\Lambda }(L)$ denote the 3-manifold we get by performing
$\Lambda _i$-surgery on $L_i$ for $i=1,\ldots , \ell$. Then, the
following holds:
\begin{thm}(\cite[Theorem~10.1]{ManOzs}) \label{thm:manozsfo} If
  $\Lambda $ is sufficiently large (that is, for all $i$ the
  coordinate $\Lambda _i\in \Z$ is sufficiently large) then the
  Heegaard Floer chain complex ${\mathbf {CF}}^-(Y_{\Lambda }(L), \s
  )$ is quasi-isomorphic to $\Chain^-({\mathcal H},\s)$. \qed
\end{thm}

Although for general links $\Chain^-({\mathcal H},\s)$ can be
challenging to compute, in the case where $L$ is the link diagram
associated to a plumbing tree, the complex $\Chain^-({\mathcal H},\s)$
can be easily determined with the help of the above theorem. Recall
that an {\em $L$-space} is a rational homology 3-sphere $Y$ with
the property that for each $\SpinC$ structure $\s$ over $Y$, the
Heegaard Floer homology ${\rm {\mathbf {HF}}} ^-(Y,\s )\cong \Field[[U]]$.

\begin{lem}\label{lem:lspaces}
  Let $G$ be a plumbing tree and $L$ be its
  corresponding link in $S^3$. Then,
  for each ${\mathbf s}\in {\mathbb H}(L)$,
  there is a homotopy equivalence
  $\Chain^-({\mathcal H},{\mathbf s})\simeq \Field[[U]]$.
\end{lem}
\begin{proof}
  By Theorem~\ref{thm:manozsfo} (which is identical to
  \cite[Theorem~10.1]{ManOzs}), $\Chain^-({\mathcal H},{\mathbf s})$
  computes the Heegaard Floer homology of a 3-manifold obtained by
  sufficiently positive surgeries on $L$.
  By~\cite[Lemma~2.6]{OSzplum}, this 3-manifold is an $L$-space,
  providing the desired isomorphism.
\end{proof}

The result given in~\cite[Theorem~7.7]{ManOzs} (restated in
Theorem~\ref{thm:SurgeryTheorem} below) provides a chain complex, described in
terms of the $\Chain^-({\mathcal H},\s)$ from above, which computes
the Heegaard Floer homology of arbitrary surgeries on $L$. To describe
this, we need a little more notation. Let us fix $\Lambda = (\Lambda
_1, \ldots , \Lambda _{\ell})$.  Let $M\subseteq L$ be a sublink with
$m$ components.  The projection map
\[
\psi^M \colon \HH(L)\to \HH(L- M)
\]
is defined as follows. Label the components of
$L=L_1\cup\dots\cup L_{\ell}$, and the components of 
$L-M=L_{j_1}\cup\dots\cup L_{j_{\ell-m}}$.  We then define
$$\psi^M=(\psi^M_{j_1},\dots,\psi^M_{j_{\ell-m}})$$
by 
$$\psi^M_{j_i}(\s)=s_{j_i}-\frac{\ell k(L_i,M)}{2}.$$
Here, $\ell k(L_i,M)$ denotes the linking number of $L_i$ with $M$;
recall that both are oriented (via an orientation induced from the
ambient link $\orL$).

As a module over $\Ring= \Field [[U_1, \ldots , U_{\ell}]]$, 
the surgery complex for the 3-manifold $Y_{\Lambda }(L)$ is defined by
\begin{equation}
  \label{eq:SurgeryComplex}
({\mathcal C}^-({\mathcal H},\Lambda),{\mathcal D}^-)
= \bigoplus_{M\subseteq L} \prod_{{\s}\in {\mathbb H}(L)}
\Chain^-({\mathcal H}^{L-M},\psi^M({\s})).
\end{equation}
To define its differential, we need yet more notation. We
need to give some algebraically defined maps, which are indexed by
sublinks $M\subseteq L$, equipped with orientations (not necessarily
agreeing with the induced orientation from $\orL$). We write this data
(sublink, together with a possibly different orientation) $\orM$; and
let $I_+(\orL,\orM)$ resp. $I_-(\orL,\orM)$ denote the sublink
consisting of components of $M$ whose orientation (in $\orM$) agree
resp. disagree with the orentation on the ambient link $\orL$. For a
sublink $M\subseteq L$, we let $\Omega(M)$ denote the set of
orientations on $M$.

Let $\oHH(L)$ denote the extension of $\HH(L)$, where we allow some of
the components to be $\pm \infty$.  For $i \in \{1, \dots, \ell\}$, we
define a projection map $p^{\orM} : \HH(L) \to \oHH(L) $ so that the $i^{th}$ component of $p^{\orM}({\mathbf s})$ is specified by
$$ 
\begin{cases}
+\infty & \text{ if } i \in I_+(\orL, \orM), \\
-\infty & \text{ if } i \in I_-(\orL, \orM), \\
s_i & \text{ otherwise.}
\end {cases}
$$

There are algebraically defined maps
\[ 
{\mathcal I}^{\orM}_{\s} : \Chain^-({\mathcal H},
{\s}) \to \Chain^- ({\mathcal H}, p^{\orM}({\s}))
\]
given by
\[ 
{\mathcal I}^{\orM}_{\s} \x = \prod_{i \in I_+(\orL, \orM)}
U_{i}^{\max(A_i(\x) - s_i, 0)} \cdot \prod_{i \in I_-(\orL,
  \orM)} U_{i}^{\max(s_i - A_i(\x), 0)} \cdot \x.
\]

For each sublink $M\subset L$ fix a Heegaard diagram ${\mathcal H}^{L-
  M}$, and fix an orientation $\orM$ on $M$.  Let $J(M)\subset
\oHH(L)$ denote the subspace $\s=(s_1,\dots,s_\ell)$, for which
$s_i=+\infty$ if $L_i\in I_+(M)$, and $s_i=-\infty$ if $L_i\in
I_-(M)$.  Counting holomorphic curves induces a homotopy equivalence
\[ 
\theta ^{\orM}_{\s}\colon \Chain^- ({\mathcal H}, p^{\orM}({\s })) \to
\Chain^- ({\mathcal H}^{L- M},\psi^{\orM}({\s})).
\] 
(This homotopy equivalence was called ${\hat D}_{\s}^{\orM}$
in~\cite{ManOzs}. We renamed it so that that it does not look like a
differential.)

The differential on the surgery complex is given as a sum 
of components
\[{\Phi}_{\s}^{\orM}\colon \Chain^-({\mathcal H},{\s})
\to \Chain^-({\mathcal H}^{L-M},\psi^{\orM}({\s})),
\]
defined by
\[
\Phi_{\s}^{\orM}=\theta^{\orM}_{p^{\orM}({\s})}\circ
{\mathcal I}^{\orM}_{\s}.
\] 
We now define the boundary operator ${\mathcal D}^-$ on the surgery
complex as follows. For $\s \in \HH(L) $ and $\x \in
\Chain^-({\mathcal H}^{L - M}, \psi^{M}(\s))$, we set
\begin {eqnarray*}
  {\mathcal D}^-(\s, \x) &=& \sum_{N \subseteq L - M} \sum_{\orN \in
    \Omega(N)}
  (\s + \Lambda_{\orL, \orN}, \Phi^{\orN}_{\psi^{M}(\s)}(\x)) \\
  &\in& \bigoplus_{N \subseteq L - M} \bigoplus_{\orN \in \Omega(N)}
  \Chain^-({\mathcal H}^{L-M-N}, \psi^{M \cup \orN} (\s)) \subseteq
  \Chain^-({\mathcal H}, \Lambda).
\end {eqnarray*}
Of course, the homotopy equivalences $\theta ^{\orM}_{\s}$ appearing
in the differential $\Phi_{\s}^{\orM}$ are, in general, tricky to
compute.  For our present purposes, though, it turns out that a
precise computation is unnecessary.

Recall that
$({\mathcal C}^-({\mathcal
  H},\Lambda),{\mathcal D}^-)$ is a module over $\Field[[U_1,\dots,U_n]]$.
Choosing $U=U_1$, we can view it as a module over $\Field[[U]]$ (it
will turn out that our results are independent of the numbering of the
$U_i$).  

The complex $({\mathcal C}^-({\mathcal H},\Lambda),{\mathcal D}^-)$
admits a natural splitting into summands, as follows.  Consider the
subspace $H(L,\Lambda)$ of $H_1(Y- L)$ spanned by framings $\Lambda_i$
of the components of $L$.  The
complex $({\mathcal C}^-({\mathcal H},\Lambda),{\mathcal D}^-)$
naturally splits into summands indexed by the quotient space
$\HH(L)/H(L,\Lambda)$. In turn, this quotient space is naturally
identified with $\SpinC(Y,\Lambda)$, via for example, the filling
construction from~\cite[Section~3.7]{OSzlinks}.

One of the key results in~\cite{ManOzs} is the following:

\begin{thm} (\cite[Theorem~7.7]{ManOzs}) 
  \label{thm:SurgeryTheorem}
  The homology of the chain complex $({\mathcal C}^-({\mathcal
    H},\Lambda),{\mathcal D}^-)$ is identified with
  ${\mathbf{HF}}^-(Y_G)$.  Indeed, the identification respects the
  splitting of both spaces into summands indexed by
  $\SpinC(Y_G)$. \qed
\end{thm}

The surgery complex has a natural filtration $S$ induced by the number
of components in the sublink $M$.  The differential ${\mathcal D}^-$
then splits as
\[
{\mathcal D}^- = \sum_{k=0}^{\infty}{\mathcal D}^-_k,
\]
where ${\mathcal D}^-_k$ is a term which drops the filtration level by
exactly $k$. In particular, ${\mathcal D}^-_0$ is the differential on
the associated graded complex. 

By the {\em $E_1$ term of the spectral sequence}, we mean the
chain complex whose underlying $\Field[[U]]$-module is $H_*({\mathcal
  C}^-({\mathcal H},\Lambda),{\mathcal D}^-_0)$, and whose
differential is induced by ${\mathcal D}^-_1$.

\begin{prop}\label{prop:ssident}
  The $E_1$ term in the filtration on $({\mathcal C}^-({\mathcal
    H},\Lambda),{\mathcal D}^-)$ is identified with $\CFmComb(G)$.
\end{prop}

\begin{proof}
  Let us first identify the $\Field [[U]]$-modules.  Recall that
  $\Vertices (G)$ can be used to index the components of $L$,
  therefore sublinks of $L$ naturally correspond subsets of
  $V=\Vertices (G)$. Furthermore, a characteristic element $K$
  specifies a $\SpinC$ structure on $X_G$, and therefore an element
  $\s \in \HH (L)$.  By Lemma~\ref{lem:lspaces}, we have that
  $H_*(\Chain^-({\mathcal H}^{L-M},\psi^M({\mathbf s})))=\Field
  [[U]]$. Mapping the generator $U^i\otimes [K,E]$ of $\CFmComb (G)$
  to $U^i$ in the factor $H_*(\Chain^-({\mathcal
    H}^{L-M},\psi^M({\mathbf s})))$ of $H_*({\mathcal C}^-({\mathcal
    H},\Lambda),{\mathcal D}^-_0) $ corresponding to the sublink $M$
  indexed by $E$ and the $\SpinC$ structure $\s$ corresponding to $K$,
  we get an isomorphism
  \[
  \CFmComb (G)\to H_*({\mathcal C}^-({\mathcal H},\Lambda),{\mathcal D}^-_0) 
  \]
  of $\Field [[U]]$-modules.

  Therefore, in order to verify the lemma, we need to identify
  ${\mathcal D}^-_1 $ with the boundary operator $\partial$ of
  $\CFmComb (G)$ described in Equation~\eqref{eq:boundary}.  Let
  $M'\subseteq M$ denote a sublink with $|M|=|M'|+1$.  The boundary
  map ${\mathcal D}^-_1$ applied to an element of
  $H_*(\Chain^-({\mathcal H}^{L-M},\psi^M({\mathbf s})))$ has two
  components in $H_*(\Chain^-({\mathcal H}^{L-M'},\psi^{M'}({\mathbf
    s})))$, which correspond to the two orientations of the knot $M-
  M'$.  Let us denote the two components $d_1^+$ and $d_1^-$. Recall
  that $M- M'$ is a knot, and hence it corresponds to some vertex $v$
  of the plumbing graph $G$. Although $M'- M$ is the unknot in $S^3$,
  in $Y_M$ it represents a possibly complicated knot, which we denote
  $K_v\subset Y_M$.

  The components $d_-$ and $d_+$ of the differential have an interpretation as a
  four-manifold invariant. Specifically, 
  the following square commutes:
  \[
  \begin{CD}
    H_*(\Chain^-({\mathcal H}^{L-M},\psi^M({\mathbf s})))
    @>{d_1^{\pm}}>>
    H_*(\Chain^-({\mathcal H}^{L-M'},\psi^{M'}({\mathbf s}))) \\
    @VVV @VVV \\
    {\mathbf {HF}}^-(Y_{M}) @>{\t _{\pm}}>> {\mathbf {HF}}^- (Y_{M'})
  \end{CD}
  \]
  Here, $Y_M$ resp. $Y_{M'}$ denotes any sufficiently large positive
  surgery on $M$ resp $M'$, the vertical maps are the identifications
  from Lemma~\ref{lem:lspaces}, the top horizontal map is either of
  the two maps $d_1^{\pm}$, and the bottom horizontal map is induced
  by the single two-handle cobordism $W$ from $Y_{M}$ to $Y_{M'}$,
  equipped with one of the two $\SpinC$ structures ${\t}_+$
  or ${\t}_-$ of maximal square.  An orientation on
  $M- M'$ specifies which component $d^\pm_1$ we are using:
  when the orientation of $M- M'$ agrees with that on $L$, we
  denote the component by $d^+_1$, and the other by $d^-_1$. 

  The orientation on $M- M'$ also specifies which of the two maximal
  square $\SpinC$ structures ${\t}_\pm$ we are using.  Both ${\t}_+$
  and ${\t}_-$ are $\SpinC$ structures with maximal square, they have
  the same evaluation on $Y_M$, and
  \[  
  {\t}_+ = {\t}_- + \PD[F],
  \]
  where here $F\in H_2(W,Y_M; \Z)\cong\Z$ is the generator with the
  property that $\partial F \in H_1(Y_M; \Z )$ corresponds to our knot $M-
  M'$ with its given orientation.  (Commutativity of the above square
  is verified in~\cite[Theorem~10.2]{ManOzs}.)

  Both of the top horizontal maps are clearly non-trivial (they are
  isomorphisms in all sufficiently large degrees), so they must both
  be multiplication by some power of $U$. We let $\alpha_v$ denote the
  $U$-power associated to $d^+_1$ and $\beta_v$ denote the $U$-power
  associated to $d^-_1$.

  \begin{lem}
    \label{lem:HFIndependence}
    The exponents $\alpha_v$ and $\beta_v$ are independent of the
    surgery coefficients $\Lambda$.
  \end{lem}

  \begin{proof}
    This is clear: the maps $d^+_1$ and $d^-_1$ make no reference to
    surgery coefficients.
  \end{proof}
  
  The same property holds on the lattice homology side:

  \begin{lem}
    \label{lem:LatticeIndependence}
    Let $G$ and $G'$ be two plumbing graphs, whose underlying graphs
    $\Gamma$ and $\Gamma'$ coincide. Fix $K\in \Char(G)$, and $E\subset
    \Vertices(G)=\Vertices(G')$, and $v\in E$.  Let $K'\in\Char(G')$ be the
    characteristic vector with 
    \begin{align*}
      K'(v)+m'_v&=K(v)+m_v \\
      K'(w)&=K(w)
    \end{align*}
    for all $w\neq v$.
    Then,
    \begin{align*}
      a_v[K,E]& =a_v[K',E] \\ 
      b_v[K,E] &= b_v[K'E]
    \end{align*}
  \end{lem}

  \begin{proof}
    By Equation~\eqref{eq:gweight} and
    the choice of $K'$, $f[K,I]=f[K',I]$.
    Since $f$ determines $a_v$ and $b_v$, the claim follows.
  \end{proof}

  \begin{lem}
    \label{lem:NegSurgeryIdent}
For sufficiently negative surgery coefficents along the sublink $M$ we
have that $a_v=\alpha_v$ and $b_v=\beta_v$
  \end{lem}
  
  \begin{proof}
If the surgery coefficients along the sublink $M'$ are sufficiently
negative, the 3-manifold $Y_{M'}$ is an $L$-space. Therefore the
statement of the lemma is essentially~\cite[Proposition~4.1]{Lspace}
(cf.~\cite[Remark~4.2]{Lspace})
applied to the graph $M'$, where $v=M- M'$ is the distinguished
vertex.
  \end{proof}

  The identification stated in the proposition is equivalent to the
  statement that $a_v=\alpha_v$ and $b_v=\beta_v$ for the given
  framing $\Lambda$. This statement, however, is an immediate
  consequence of Lemmas~\ref{lem:HFIndependence},
  \ref{lem:LatticeIndependence}, and~\ref{lem:NegSurgeryIdent}.
\end{proof}

\begin{prop}\label{prop:MaslovEqual}
  The identification of Proposition~\ref{prop:ssident} 
  respects the (relative or absolute, depending on 
  the $\SpinC$ structure) Maslov gradings.
\end{prop}
\begin{proof}
  Recall that in the proof of Proposition~\ref{prop:ssident} the
  generator $[K,E]$ of $\CFmComb (G)$ has been identified with the
  pair $\s , M$, where $M$ is a sublink of the link $L$ defined by the
  plumbing graph $G$ and $\s$ is a relative $\SpinC$ structure. In
  particular, $\vert M \vert =\vert E \vert$.  We claim that this
  identification respects Maslov gradings. Indeed, if $K$
  represents a torsion $\SpinC$ structure, then the absolute Maslov
  grading of $[K,\emptyset]$ (thought of as an element of
  $\CFmComb(G)$) coincides with that of $(\s,\emptyset )$ (though of as an
  element of $({\mathcal C}^-({\mathcal H},\Lambda),{\mathcal
    D}^-)$). Since the boundary map drops Maslov grading by one, the
  identification of Maslov gradings extends to all generators of the
  form $[K,E]$.
  
  The same argument applies in the relatively graded setting (when $K$
  restricts to a non-torsion class on $\partial X_G=Y_G$).
\end{proof}

We turn to Theorem~\ref{thm:mainss}:

\begin{proof}[Proof of Theorem~\ref{thm:mainss}]
  Theorem~\ref{thm:SurgeryTheorem} presents ${\mathbf{HF}}^-(Y_G)$ as
  the homology of a filtered chain complex.  Theorem~\ref{thm:mainss}
  now follows from this theorem, together with the interpretation of
  the $E_1$ term on the filtration provided by
  Proposition~\ref{prop:ssident}. Proposition~\ref{prop:MaslovEqual}
  then provides the proof of the claim about the identification of
  Maslov gradings.
\end{proof}

Certain higher differentials in the spectral sequence vanish for {\em
  a priori} reasons. This is most easily seen when one appeals to
gradings.  
\begin{prop}\label{prop:d2n}
The differential ${\mathcal {D}}^-_{2n}$ on the page $E_{2n}$ vanishes.
\end{prop}
\begin{proof}
  Note first that all differentials on $E_r$ drop Maslov grading by
  $1$, and in particular change the Maslov grading by $1\pmod{2}$
  (see Lemma~\ref{lem:DiffDropsOne}).
  Looking at the expression of the grading $\gr$ on lattice homology,
  we see that the relative Maslov grading $\pmod{2}$ of any element
  $U^i\otimes[K,E]$ agrees with $|E|\pmod{2}$. Moreover, ${\mathcal
    {D}}^-_{k}$ drops $|E|$ by $k$. It follows from these observations
  and the identification of the Maslov gradings on the two theories
  (given by Proposition~\ref{prop:MaslovEqual}) that $D^-_{2n}$
  vanishes.
\end{proof}

\subsection{Module structures and the spectral sequence}
After establishing Theorem~\ref{thm:mainss}, we need a slight further
refinement in order to provide the proof of
Corollary~\ref{c:type2}.  

Suppose all the higher differentials on the spectral sequence
appearing in Theorem~\ref{thm:mainss} vanish. Even in this case we
cannot necessarily conclude that ${\mathbf{HF}}^-(Y_G)$ is computed by
lattice homology: rather, ${\mathbf{HF}}^-(Y_G)$ is determined up to
extensions. This allows us to identify the two theories only as vector
spaces over $\Field$, but not as $\Field [U]$-modules. In certain
cases, this indeterminacy can be removed by working with coefficients
in $\Field[U]/U^n$ for all $n$. In the rest of the section we spell
out the details of this observation.

The complex $({\mathcal C}^{[n]}({\mathcal H},\Lambda),{\mathcal
  D}_{[n]}^-)$ will denote the complex over $\Field[U]/U^n$ defined by
taking the complex defined in Proposition~\ref{prop:ssident},
$({\mathcal C}({\mathcal H},\Lambda),{\mathcal D}^-)$, and setting
$U^n=0$.  (Recall that we viewed $({\mathcal C}({\mathcal
  H},\Lambda),{\mathcal D}^-)$ as a module over $\Field[[U]]$ by
defining the action by $U$ to be multiplication by $U_1$. To view it
as a module over $\Field[U]/U^n$, we must set $U_1^n=0$.) The complex
$({\mathcal C}^{[n]}({\mathcal H},\Lambda),{\mathcal D}_{[n]}^-)$
naturally inherits a filtration from $({\mathcal C}({\mathcal
  H},\Lambda),{\mathcal D}^-)$.

\begin{lem}
  Fix any positive integer $n$, and consider the spectral sequence on
  $({\mathcal C}^{[n]}({\mathcal H},\Lambda),{\mathcal D}_{[n]}^-)$
  induced from its filtration. This spectral sequence has
  $E_1$-term isomorphic to $\CFaComb ^{[n]} (G)$. \qed
\end{lem}
\begin{proof}
 This is true because (thanks to Lemma~\ref{lem:lspaces}) the $E_1$
 term $H_*({\mathcal C}({\mathcal H},\Lambda),{\mathcal D}^-_0)$ is
 torsion free, as an $\Field[[U]]$-module. More explicitly, consider the
 filtered chain complex $C=({\mathcal C}({\mathcal
   H},\Lambda),{\mathcal D}^-)$. The associated spectral sequence has
 $E_1=H_*(C,{\mathcal D}^-_0)$, equipped with the differential induced
 by the ${\mathcal D}^-_0$-chain map ${\mathcal D}^-_1$.
  
  The filtered chain complex $C'=({\mathcal C}^{[n]}({\mathcal
    H},\Lambda),{\mathcal D}^-_{[n]})$ is gotten from $C$ by $C\otimes
  \Field[U]/U^n$.  In general, its $E_1$ term is computed by
\[  
E_1(C')=H_*(C,{\mathcal D}^-_0)\otimes \Field[U]/U^n)\oplus
  \Tor(H_*(C,{\mathcal D}^-_0)\otimes \Field[U]/U^n),
\]
 converging to $H_*(C')$.  In the case at hand, though,
 $H_*(C,{\mathcal D}^-_0)$ is a direct product of Heegaard Floer
 homology groups of 3-manifolds obtained as large surgeries on various
 components of our link, each of which, according to
 Lemma~\ref{lem:lspaces}, contributing a factor of $\Field[[U]]$.
 Since $\Tor(\Field[[U]],\Field[U]/U^n)=0$, we have that
 $\Tor(H_*(C,{\mathcal D}^-_0)\otimes \Field[U]/U^n)=0$. It follows
 that $E_1(C')=E_1(C)\otimes \Field[U]/U^n$, equipped with the
 differential induced from ${\mathcal D}^-_1$. But this $E_1$ term is
 precisely $\CFaComb^{[n]}(G)$.
\end{proof}

Now the version of Theorem~\ref{thm:mainss} for the $U^n=0$ truncated
theory has the following shape.
\begin{thm}
  \label{thm:UnSpecialized}
  Suppose that $G$ is a plumbing tree of spheres,
  and let  $Y_G$ be the corresponding $3$-manifold.  Then there is a spectral
  sequence $\{ E_i\} _{i=1}^{\infty}$ with the property that
\begin{itemize}
\item the $E_2$-term of the spectral sequence is isomorphic to the
  $U^n=0$-specialized lattice
homology ${{\HFaComb}^{[n]} }(G)$ and 
\item the spectral sequence converges to the $U^n=0$-specialized
  Heegaard Floer homology group ${\widehat{\rm {\mathbf {HF}}}} ^{[n]}
  (Y_G)$. \qed
\end{itemize}
\end{thm}

Theorem~\ref{thm:UnSpecialized} can be used to gain a little more
information about the $\Field[[U]]$-module structure on $\HHFm(Y_G)$
(in terms of lattice homology).  This improvement rests on the
following algebraic result.

\begin{lem}\label{lem:juennek}
  Suppose that $C$ and $C'$ are two Maslov-graded chain complexes over
  $\Field[[U]]$ whose homologies are finitely generated (as
  $\Field[[U]]$-modules).  If for all $n\geq 1$,
  $$H_*(C\otimes\Field[U]/U^n)\cong H_*(C'\otimes\Field[U]/U^n)$$
  as $\Field$-vector spaces, then it follows that 
  $H_*(C)\cong H_*(C')$ as $\Field[[U]]$-modules.
\end{lem}
\begin{proof}
  Fix a rational number $d$ and an integer $k>0$.
  Let $M(d,k)$ denote the Maslov-graded $\Field[[U]]$-module with the
  following two properties:
  \begin{itemize}
  \item $M(d,k)\cong \Field[U]/U^k$ as an $\Field[[U]]$-module, and
  \item the generator of $M(d,k)$ has Maslov grading $d$ (i.e. the
    whole module is supported in Maslov gradings between $d$ and
    $d-2k$).
  \end{itemize}
  We extend the definition of $M(d,k)$ to $k=0$ to be
  the Maslov-graded $\Field[[U]]$-module
  with the following two properties:
  \begin{itemize}
  \item $M(d,0)\cong \Field[[U]]$ as an $\Field[[U]]$-module, and 
  \item the generator of $M(d,0)$ has Maslov grading $d$ (i.e. the
    whole module is supported in Maslov gradings $\leq d$).
  \end{itemize}

  Since $\Field[[U]]$ is a principal ideal domain, the
  finitely-generated, Maslov-graded $\Field [[U]]$-module $H_*(C)$
  splits as the direct sum of modules of the form $M(d,k)$; i.e.
  \[
H_*(C)\cong \bigoplus_{d\in\Q,k\in\{0,\dots,\infty\}}
M(d,k)^{c_{d,k}},
\]
 where $c_{d,k}$ is a collection of non-negative integers, only
 finitely many of which are positive.

  Our goal is to show that the collection of $\Field$-vector spaces
  $\{H_*(C\otimes \Field[U]/{U^n})\}_{n=0}^{\infty}$ uniquely
  determines the isomorphism type of $H_*(C)$ as an
  $\Field[[U]]$-module, i.e. it uniquely determines the coefficients
  $\{ c_{d,k}\} _{d,k}$.
 
  This statement follows from an application of the universal coefficients
  theorem,
  stating that
  \begin{equation}
    \label{eq:UniversalCoefficients}
    H_*(C\otimes \Field[U]/U^n)\cong (H_*(C)\otimes \Field[U]/U^n)
    \oplus \Tor_{*-2k-1}(H_*(C),\Field[U]/U^n),
  \end{equation}
  where the perhaps unfamiliar shift in grading (of $2k+1$, rather
  than simply $1$) on the $\Tor$ results from the fact that the action
  by $U$ shifts Maslov grading by $2$.
  
  We find it convenient to encode the input data in terms of a
  two-variable generating function
  \[
P_{C}(s,t)=\sum_{n\geq 0, m} \dim _{\Field }H_{m}(C\otimes
\Field[U]/U^n) s^n t^{-m}.
\]
  By
  Equation~\eqref{eq:UniversalCoefficients}, 
\[
P_{H_*(C)}=\sum_{d,k}
  c_{d,k}\cdot P_{M(d,k)},
\]
 where
  \[
P_{M(d,k)}= \sum_{n\geq 0,m} \dim _{\Field} (M_{m}(d,k)\otimes
\Field[U]/U^n) s^n t^{-m} +
\]
\[
  t^{2k+1}\sum_{n\geq 0,m} \dim _{\Field }(M_{m}(d,k)\otimes
  \Field[U]/U^n) s^n t^{-m}.
\]
  The lemma is proved once we show that the functions $P_{M(d,k)}$ are
  linearly independent (over $\Z$).
  This in turn follows form a straightforward calculation:
  \[P_{M(d,k)}=t^{-d}
  \left(\frac{\sum_{i=0}^{k-1} (st^2)^i}{1-s}+
    t^{2k+1}\frac{\sum_{i=0}^{k-1} s^i}{1-st^2}\right)=
  t^{-d}
  \left(\frac{1-(st^2)^k + t^{2k+1} (1-s^k)}{(1-s)(1-st^2)}\right) ,
  \]
  thought of as a rational function in $s$; and also
  \[
  P_{M(d,0)}=\frac{t^{-d}}{(1-s)(1-st^2)}.
  \]
  Note that $(1-s)(1-st^2) P_{M(d,k)}$ is a degree-$k$ polynomial in
  $s$. When $k>0$, the coefficient of $s^k$ is
  $-t^{-d}(t^{2k}+t^{2k+1})$, while at $k=0$, we get the constant (in
  $s$) polynomial $t^{-d}$. The linear independence of the
  $P_{M(d,k)}$ follows immediately.
\end{proof}

\begin{rem}\label{rem:MaslovNonTorsion}
  A version of Lemma~\ref{lem:juennek} applies when $C$ and $C'$ are
  two relatively $\Zmod{d}$ Maslov-graded chain complexes, as well. In
  that case, the generating function $P_{M}(s,t)$ is defined over
  $\Z[\Zmod{d}][[s]]$; i.e. $t$ is a primitive $n^{th}$ root of unity.
\end{rem}

\begin{cor}
  \label{cor:Equals}
  Suppose that all higher differentials $D_i^-$ vanish for $i\geq 2$
  in the spectral sequence associated to $({\mathcal C}^-({\mathcal
    H},\Lambda),{\mathcal D}^-)$. Suppose that the same holds for all
  the truncated spectral sequences $({\mathcal C}^{[n]}({\mathcal
    H},\Lambda),{\mathcal D}^-_{[n]})$.
  Then, $\HFmComb(G)$ and $\HHFm(Y_G)$ are isomorphic as $\Field[[U]]$-modules.
\end{cor}

\begin{proof}
  This follows quickly from Lemma~\ref{lem:juennek}.
\end{proof}

\section{Graphs of type-2}
\label{sec:appl}

The proof of Corollary~\ref{c:type2} relies on the following simple
corollary of the existence of the surgery triangle for lattice
homology.  (The exact sequence we will use in the proof has been
described by Greene \cite[Theorem~3.1]{Josh}, and independently by
N\'emethi \cite{latticetriangle}; see also the Appendix for a version
adapted to the present notational conventions.)
\begin{thm}\label{thm:type}
Suppose that the plumbing tree $G$ is of type-$k$. Then $\HFmComb _q (G)=0$
for $q>k$. 
\end{thm}
\begin{proof}
The proof of the theorem proceeds by induction on $k$. For $k=0$
(i.e. if $G$ is rational), the claim follows from
\cite[Proposition~4.1.4.]{lattice}. Suppose now that $G$ is of
type-$(k+1)$ and assume that the claim of the theorem holds for graphs
of type at most $k$.  Let $v$ be a vertex of $G$ from the set $\{
v_{i_1}, \ldots , v_{i_{k+1}}\}$ appearing in
Definition~\ref{def:type} of the type of $G$. Then, by the same
definition, $G-v$ is of type-$k$.  Let $G_{-n}$ denote the graph we
get from $G$ by decreasing the framing of the chosen $v$ by $n\in
{\mathbb {N}}$. If $n$ is sufficiently large, then (again by
Definition~\ref{def:type}) the graph $G_{-n}$ is of type-$k$. Fix now
$q>k+1$ and consider the following portion of the long exact sequence
associated to $(G_{-n}, v)$
(cf. Corollary~\ref{cor:LongExactLattice}):
\[
\ldots \to \HFmComb _q (G_{-n})\to \HFmComb _q (G_{-n+1})\to \HFmComb _{q-1} (G-v) \to \ldots 
\]
By the inductive assumption, the first and the third terms vanish,
hence by exactness so does the middle term. Iterating this argument
until we get the given framing on $v$, the result follows and shows
that $\HFmComb _q(G)=0$ for $q>k+1$.
\end{proof}
In a similar manner, we get
\begin{thm}\label{thm:TruncatedType}
If the plumbing tree $G$ is of type-$k$ then
$\HFaComb ^{[n]}_q(G)=0$ for all $q>k$ and all $n\in \N$. \qed 
\end{thm}
\begin{rem}
For $G$ negative definite, Theorem~\ref{thm:type} can be sharpened to
$q\geq k$.  This strengthening, however, does not hold for the
truncated theories $\HFaComb ^{[n]}(G)$, hence the negative definite
assumption does not improve Theorem~\ref{thm:TruncatedType} in this sense.
\end{rem}

From these results the proof of the corollary is a simple
exercise:
\begin{proof}[Proof of Corollary~\ref{c:type2}]
  Suppose that $G$ is a plumbing tree (or forest) of type-2 and
  consider the spectral sequence provided by
  Theorem~\ref{thm:mainss}. By Proposition~\ref{prop:d2n} we have that
  $E_2=E_3$, and since by Theorem~\ref{thm:type} the homology (and so
  the $E_2$-table of the spectral sequence) concentrates on the rows
  with $\vert E \vert$-gradings $0,1,2$, the higher differentials
  point from or to vanishing groups, implying that ${\mathcal
    {D}}^-_i=0$ for all $i\geq 3$. This means that $E_2=E_{\infty}$,
  hence by Theorem~\ref{thm:mainss} the lattice and Heegaard Floer
  homologies coincide, as vector spaces over $\Field$. To get the
  corresponding isomorphism as $\Field[[U]]$-modules, we use the
  version of the spectral sequence over $\Field[U]/U^n$,
  Theorem~\ref{thm:UnSpecialized} cf. Corollary~\ref{cor:Equals}. For
  torsion $\SpinC$ structures, the isomorphism of Maslov-graded
  $\Field[[U]]$-modules follows now from Lemma~\ref{lem:juennek}. For
  non-torsion $\SpinC$ structures, we appeal to the modification of
  the proof of Lemma~\ref{lem:juennek} described in
  Remark~\ref{rem:MaslovNonTorsion}.
\end{proof}

\section{Appendix: the exact sequence}
\label{sec:app}
 
For completeness, in this final section we prove the exact sequences
in lattice homology used above.  These results could be derived from
\cite{Josh, latticetriangle}, but we find it convenient to include
this proof here, as it follows the conventions and formalism
introduced in Section~\ref{sec:review}.

Let $G$ be a plumbing graph, and $v\in\Vertices(G)$ be a distinguished
vertex with framing $m_v$. $G-v$ will denote the graph obtained by
omitting the vertex $v$.  We define the \emph{extension map} $ \Phi
_v\colon \CFmComb(G-v)\to \CFmComb(G)$ by the formula
\begin{equation}\label{eq:phiformula}
\Phi_v([K,E]) = \sum_{p\equiv
m_v\pmod{2}} [(K,p),E].
\end{equation}
On the right-hand-side we write characteristic vectors for $G$ as
pairs $(K,p)$, where $K$ is a characteristic vector for $G-v$, and $p$
is the evaluation of the characteristic vector on the distinguished
vertex $v$.  Since any component of $\Phi _v ([K,E])$ determines
$[K,E]$, it is easy to see that the above formula indeed provides a
function on $\CFmComb$ (meaning that any component of $\Phi _v(x)$ for
a possibly infinite sum $x$ has coefficient in $\Field [[U]]$). In
fact, the above principle also shows that $\Phi _v$ is injective.

\begin{lem}
\label{lemma:PhiIsChain}
For each vertex $v\in \Vertices(G)$, the map $\Phi_v$ is a chain map.
\end{lem}

\begin{proof}
  This follows immediately from the fact that for any $E\subset G-v$,
  the $(G-v)$-weight $f_{G-v}[K,E]$ of the pair $[K,E]$ agrees with
  the $G$-weight $f_G[(K,p),E]$ of the pair $[(K,p),E]$ where $p$ is
  any integer with the allowed parity. (Here $f_{G-v}$ and $f_G$ refer
  to the function defined in Equation~\eqref{eq:gweight} with the
  respective graphs $G-v$ and $G$.) This implies that the
  corresponding functions $g_{G-v}$ and $g_G$ of minimal weights also
  coincide, and since the boundary maps are determined by these
  minimal weight functions, the result follows at once.
\end{proof}

Let $G_{+1}(v)$ denote the graph $G$ with the same framings, except on
the vertex $v$ we consider $m_v+1$ instead of $m_v$. Define the map
$\Psi_v\colon \CFmComb(G) \to \CFmComb(G_{+1}(v))$
by the formula 
\begin{equation}\label{eq:psivdef}
\Psi_v[(K,p),E]=\sum _{m=-\infty}^{\infty}U^{s_m}\otimes [(K,p+2m-1), E],
\end{equation}
where $s_m=g_{G_{+1}(v)}[(K,p+2m-1),E]-g_G[(K,p),E]+\frac{m(m-1)}{2}$.
It is easy to see that when $v\not \in I\subset E$ the equality
$f_{G_{+1}(v)}[(K,p+2m-1),I]=f_G[(K,p),I]$ holds, hence
$s_m=\frac{m(m-1)}{2}\geq 0$ in this case.  If $v\in I\subset E$ then
$f_{G_{+1}(v)}[(K,p+2m-1), I]-f_G[(K,p),I]$ is at most $\vert m\vert $
in absolute value, hence after addig $\frac{m(m-1)}{2}$ to it, the
result will be nonnegative.  In conclusion, $s_m$ is nonnegative for
any $(K,p)$ and $m$.  

Once again, a short argument is needed to confirm that the above
formula defines a function on $\CFmComb$, that is, for an 
infinite sum $\sum _{i\in \Z }U^{m_i}[K_i,E_i]$ all coordinates of the
image admit a coefficient in $\Field [[U]]$. This property follows from 
the fact that if $p+2m$ is fixed then the value $s_m$ converges to infinity
as $m\to \pm \infty$, implying that at most finitely many terms $[(K,p+2m-1),E]$
with $p+2m$ fixed can have a given $U$-power in the image. 
\begin{lem}\label{lem:psichainmap}
The map $\Psi _v $ is a chain map.
\end{lem}
Before starting the proof of this lemma, we need to define one further map.
Suppose that the graph $G_e$ is constructed from $G$ by adding a new vertex
$e$ with framing $(-1)$ and an edge connecting $e$ and $v$. Consider the map
\[
P \colon \CFmComb(G_e)\longrightarrow
		\CFmComb(G_{+1}(v))
\]
given by the formula: 
\[
P[(K,p,2m-1),E]= \left\{\begin{array}{ll}
U^{s}\otimes [(K,p+2m-1),E] & {\text{if $e\not\in E$}} \\
0	& {\text{if $e\in E$,}} 
\end{array}\right.
\]
where
$s_m=g_{G_{+1}(v)}[(K,p+2m-1),E]-g_{G_e}[(K,p,2m-1),E]+\frac{m(m-1)}{2}$.
(Once again, $(K,p, 2m-1)$ denotes the cohomology class on $G_e$ which
is $K$ on $G-v$, takes the value $p$ on $v$ and the value $2m-1$ on
$e$.)  As above, it can be verified that $P$ extends to a well-defined
function on $\CFmComb (G_e)$.
\begin{lem}
\label{lemma:PisChain}
The map $P$ is a  chain map.
\end{lem}
\begin{proof}
We wish to prove 
that $\partial\circ P[(K,p,2m-1),E] = P\circ \partial [(K,p,2m-1),E].$
First, we consider the case where $e\in E$.  In this case the
left hand side is zero. Moreover,
\begin{eqnarray*}
P\circ \partial [(K,p,2m-1),E] &=&
P(U^{a_e[(K,p,2m-1),E]}\otimes [(K,p,2m-1),E-e]) \\
&& +
P(U^{b_e[(K,p,2m-1),E]}\otimes [(K,p+2,2m-3),E-e]) \\
&=& U^{d_1}\otimes [(K,p+2m-1),E-e] + U^{d_2}  \otimes [(K,p+2m-1),E-e],
\end{eqnarray*}
where $$d_1=
a_e[(K,p,2m-1),E]+g[(K,p+2m-1),E-e]-g[(K,p,2m-1),E-e]+\frac{m(m-1)}{2}$$
and $$d_2=
b_e[(K,p,2m-1),E]+g[(K,p+2m-1),E-e]-g[(K,p+2,2m-3),E-e]+
2m^2-6m+4.$$
In fact, it is easy to see that 
$$d_1=g[(K,p+2m-1),E-e]-g[(K,p,2m-1),E]+\frac{m(m-1)}{2}=d_2,$$
so the two terms cancel.

Next, suppose that $e\not\in E$. 
Observe that
\begin{eqnarray*}
\lefteqn{P\circ \partial [(K,p,2m-1),E]=} \\
&& 
\sum_{w\in E} U^{c_1(w)}\otimes [(K,p+2m-1),E-w] 
+ U^{d_1(w)} \otimes[(K,p+2m-1)+2w^*,E-w],
\end{eqnarray*}
and 
\begin{eqnarray*}
\lefteqn{\partial \circ P[(K,p,2m-1),E]=} \\
&& \sum_{w\in E} U^{c_2(w)}\otimes [(K,p+2m-1),E-w] 
 + U^{d_2(w)}\otimes [(K,p+2m-1)+2w^*,E-w],
\end{eqnarray*}
In fact, it is easy to see that
$$c_1(w)=
g[(K,p+2m-1),E-w]-g[(K,p,2m-1),E]+\frac{m(m-1)}{2}=
 c_2(w)$$
and
$$d_1(w)=
g[(K,p+2m-1),E-w]-g[(K,p,2m-1),E]+\frac{m(m-1)}{2} + \frac{L(w)+w\cdot
  w}{2} =d_2(w),$$ where $L=(K,p+2m-1)$ and $w\cdot w$ is taken in
$G_{+1}(v)$.  This completes the verification of the statement of the
lemma.
\end{proof}
\begin{proof}[Proof of Lemma~\ref{lem:psichainmap}]
Consider now the map $\Phi _e\colon \CFmComb (G)\to \CFmComb (G_e)$.
The map $\Psi _v$ is simply the composition $P\circ \Phi _e$, and since
both maps are chain maps, so is $\Psi _v$, concluding the proof of the
lemma.
\end{proof}

\begin{thm}
\label{thm:ShortExact}
For any $v\in G$, the $U$-equivariant maps $\Psi_v$ and $\Phi_v$ fit
into a short exact sequence of chain complexes:
\begin{equation}\label{eq:minusses}
0\longrightarrow \CFmComb(G-v) \stackrel{\Phi_v}{\longrightarrow}
\CFmComb(G) \stackrel{\Psi_v}{\longrightarrow}
\CFmComb(G_{+1}(v)){\longrightarrow} 0.
\end{equation}
\end{thm}

The theorem could be proved by a direct check of exactness at each
term --- we rather choose an alternative way of first dealing with the
$U=0$ theory (and the corresponding result there) and then apply
abstract reasoning to verify the theorem.  Define the map $\Phia
_v\colon \CFaComb (G-v)\to \CFaComb (G)$ corresponding to $\Phi _v$ 
by the same formula as given by \eqref{eq:phiformula}. Next define
$\Psia _v \colon \CFaComb (G) \to \CFaComb (G_{+1}(v))$ (corresponding to the map
$\Psi _v$) by the same formula as given in Equation~\eqref{eq:psivdef}, after setting 
$U=0$.

\begin{lem}
  The map $\Psia _v\colon \CFaComb (G)\to \CFaComb (G_{+1}(v))$
  corresponding to $\Psi _v$ in the $U=0$ theory is given by the
  formula
\[
\Psi_v([(K,p),E])=[(K,p+1),E]+[(K,p-1),E]
\]
if $v\not \in E$ and by 
\begin{eqnarray*}
\lefteqn{\Psia_v([(K,p),E])} \\
&=&
\left\{\begin{array}{cl}
{[(K,p+1),E]}+{[(K,p-1),E]} 
& {\text{if $A_v([(K,p),E])< B_v([K,p],E)$,}} \\
{[(K,p+1),E]}+{[(K,p-1),E]} +{[(K,p-3),E]} 
& {\text{if $A_v([(K,p),E])= B_v([K,p],E)$,}} \\
{[(K,p-1),E]}+{[(K,p-3),E]} 
& {\text{if $A_v([(K,p),E])> B_v([K,p],E)$.}} 
\end{array}\right.
\end{eqnarray*}
for $v\in E$.
\end{lem}
\begin{proof}
Indeed, if $v\not \in E$ we have
$g_{G_1}[(K,p+2m-1),E]-g_G([(K,p),E]=0$, hence $s_m
=\frac{m(m-1)}{2}$, which is positive unless $m=0,1$, hence provides
only the two corresponding terms in the $U=0$ theory.

The case of $v\in E$ requires
a little more care. Suppose first that $A_v([(K,p),E])< B_v([K,p],E)$,
meaning that the value $g[(K,p),E]$ is taken on a subset  $I\subset E$
which does not contain $v$. Therefore for nonnegative $m$ the difference
of the $g$-functions is zero, hence $s_m=0$ implies $\frac{m(m-1)}{2}=0$,
which holds exactly when $m=0,1$, providing the two terms in the expression.
For $m<0$ and $B_v([K,p],E)-A_v([(K,p),E])=k>0$ the value of $s_m$ is 
$m+k+\frac{m(m-1)}{2}$, which is strictly positive for any $m<0$ (since
$k\geq 1$).

Suppose now that  $A_v([(K,p),E])= B_v([K,p],E)$. In this case for
$m\geq 0$ the difference of the $g$-functions is zero, hence
$s_m=0$ is equivalent with $\frac{m(m-1)}{2}=0$, providing the two terms
corresponding to $m=0,1$. For negative $m$ the term $s_m$ is equal to 
$m+\frac{m(m-1)}{2}$, and this is zero exactly when $m=-1$, giving the third
term in the expression.

Finally if $A_v([(K,p),E])> B_v([K,p],E)$ then for $m>0$ the difference
of the $g$-functions is positive (and $\frac{m(m-1)}{2}$ is nonnegative),
while for $m\leq 0$ the value of $s_m$ is equal to $m+\frac{m(m-1)}{2}$,
which is zero exactly when $m=0,-1$, giving the claimed two terms in this case.
\end{proof}

Having these formulae, now it is easy to see that the short sequence
of~\eqref{thm:ShortExact} given by the maps on the $U=0$ theory
is exact, providing the long exact sequence on homologies:

\begin{prop}
\label{prop:ShortExacta}
For any $v\in G$, the maps $\Phia _v$ and $\Psia _v$ fit into the  short exact
sequence 
\begin{equation}\label{eq:hatses}
0\longrightarrow \CFaComb(G-v)
\stackrel{\Phia _v}{\longrightarrow}
\CFaComb(G) \stackrel{\Psia _v}{\longrightarrow}
\CFaComb(G_{+1}(v))\longrightarrow 0
\end{equation}
of chain complexes.
\end{prop}
\begin{proof}
  Each group $\CFaComb(G-v)$, $\CFaComb(G)$, and $\CFaComb(G_{+1})$
  splits into a direct product indexed by pairs $K\in\Char(G-v)$,
  $E\subset\Vertices(G)$. The maps $\Phia _v$ and $\Psia _v$ obviously
  respect this splitting. We claim that these maps fit into short
  exact sequences for each summand.

  More precisely, in the case where $v\not\in E$, the corresponding
  summand of $\CFaComb(G-v)$ is one-dimensional, generated by the
  element $[K,E]$, and the desired short exact sequence is:
\[
0 \longrightarrow \Field [K,E]\stackrel{\phi _v}{\longrightarrow}
\prod_{i\in\Z}\Field[(K,k+2i),E] \stackrel{\psi _v}{\longrightarrow}
\prod_{i\in\Z}\Field[(K,k+2i+1),E] \longrightarrow 0,
\]
where 
\[
\phi _v ([K,E])=\sum_{i\in\Z}[(K,k+2i),E]
\]
and 
\[
\psi _v([(K,p),E])=[(K,p+1),E]+[(K,p-1),E].
\]
A right inverse for $\psi _v $ is determined by
$$r[(K,p-1),E]=\sum_{i=0}^{\infty}[(K,p+2i),E],$$
and it is easy to see that $\ker \psi _v=$Im$\phi _v$.

In the case where $v\in E$, we declare the corresponding summand of
$\CFaComb(G-v)$ to be trivial, so we claim that the corresponding sequence
\[
0 \longrightarrow 0 \stackrel{\phi _v}{\longrightarrow}
\prod_{i\in\Z}\Field[(K,k+2i),E] \stackrel{\psi _v}{\longrightarrow}
\prod_{i\in\Z}\Field[(K,k+2i+1),E] \longrightarrow 0
\]
is short exact, i.e. $\psi _v$ is an isomorphism. Indeed,
the map
\[
q\colon\prod_{i\in\Z}\Field[(K,k+2i+1),E] \longrightarrow 
\prod_{i\in\Z}\Field[(K,k+2i),E],
\]
which is uniquely determined by:
\[
q([(K,p-1),E])=
\left\{\begin{array}{ll}
\sum_{i=0}^{\infty}[(K,p+2i),E] &{\text{if 
$A_v([(K,p),E])<B_v([(K,p),E])$}} \\ & \\
\sum_{i=-\infty}^{0}[(K,p+2i),E] &{\text{if 
$A_v([(K,p),E])>B_v([(K,p),E])$}} \\ & \\
\sum_{i=-\infty}^{\infty}[(K,p+2i),E] &{\text{if 
$A_v([(K,p),E])=B_v([(K,p),E])$.}} 
\end{array}
\right.
\]
provides an inverse for $\psi _v$. Indeed, the fact that $\psi _v\circ
q$ and $q\circ \psi _v$ are both equal to the (respective) identities
follows from the principle, that for a given $[K,E]$ there is exactly
one value of $p$ for which $A_v[(K,p),E]=B_v[(K,p),E]$.
 \end{proof}
The short exact sequences then induce a long exact sequence on homologies,
and since both $\Phia _v$ and $\Psia _v$ respect the grading of $[K,E]$
induced by $\vert E\vert$, we get the following
\begin{cor}
  The short exact sequence of Proposition~\ref{prop:ShortExacta} induces a
  long exact sequence
\[
\ldots \to \HFaComb _{i+1} (G_{+1}(v)) \to \HFaComb _i (G-v)\to
\HFaComb _i(G) \to \HFaComb _i (G_{+1}(v))\to \
\HFaComb _{i-1}(G-v) \to \ldots
\]
on $\delta$-graded lattice homology. \qed
\end{cor}
With the above result at hand we return to the theory over $\Field [[U]]$.
\begin{proof}[Proof of Theorem~\ref{thm:ShortExact}]
First we claim that $\Psi_v\circ \Phi_v=0$. This follows from the fact that
\[
(\Psi _v\circ\Phi_v)[K,E]=\sum_p\sum_m U^{\frac{m(m-1)}{2}} \otimes [(K,p+2m-1),E].
\]
(Notice that since $v\not \in E$, we have that
$g_{G_{+1}(v)}[(K,p+2m-1),E]=g_G[(K,p),E]$, hence
$s_m=\frac{m(m-1)}{2}$.)  Observe that each term in the above sum
appears exactly twice: the term corresponding to $(p,m)$ agrees with the term
corresponding to $(p+4m-2,-m+1)$. Indeed, the system 
\[
p+2m-1=p'+2m'-1 \qquad {\mbox {and}} \qquad m(m-1)=m'(m'-1)
\]
has exactly the two solutions for $(p',m')$ given above.  This
cancellation then shows that $(\Psi _v\circ\Phi_v)[K,E]=0$, verifying
the claim.

Now we define two homology theories associated to the pair $(G,v)$:
let ${\widehat {\rm {H}}}_{SES}(G,v)$ denote the homology of the short
exact sequence \eqref{eq:hatses} (viewed it as a chain complex with
underlying group the sum of the terms in the sequence and boundary map
equal to the maps in the sequence). Similarly, ${\rm
  {H}}^-_{SES}(G,v)$ will denote the homology of the sequence
\eqref{eq:minusses}. (Since the compositions of conscutive maps in
these sequences are zero, these homologies are defined.) The content
of Proposition~\ref{prop:ShortExacta} is that ${\widehat {\rm
    {H}}}_{SES}(G,v)=0$, while in Theorem~\ref{thm:ShortExact} we want
to show that ${\rm {H}}^-_{SES}(G,v)=0$. The two homologies are,
however, connected by the Universal Coefficient Theorem.  Indeed,
${\rm {H}}^-_{SES}(G,v)$ is defined over the ring $\Field [[U]]$,
while the chain complex defining ${\widehat {\rm {H}}}_{SES}(G,v)$ can
be given from \eqref{eq:minusses} by considering the tensor product of
the $\CFmComb$-modules with $\Field$ over $\Field [[U]]$, where a
power series in $\Field [[U]]$ acts through its constant term on
$\Field$.  By the Universal Coefficient Theorem \cite{spanier} (and by
the fact that $\Field$ is a field) we get that $ {\rm
  {H}}^-_{SES}(G,v)\otimes _{\Field [[U]]}\Field = {\widehat {\rm
    {H}}}_{SES}(G,v)=0$.  Since $\Field [[U]]$ is a principal ideal
domain, the tensor product of any nontrivial module with $\Field$
(over $\Field [[U]]$) is nontrivial: consider a nontrivial element
$x\in {\rm {H}}^-_{SES}(G,v)$ and observe that the submodule generated
by it is isomorphic to $\Field [[U]]/(f(U))$ with $f(0)=0$ (since
$\Field [[U]]$ is a PID), and $\Field [[U]]/(f(U)) \otimes _{\Field
  [[U]]}\Field =\Field \neq 0$.  Since we showed that ${\widehat {\rm
    {H}}}_{SES}(G,v)=0$, this last observation then implies that ${\rm
  {H}}^-_{SES}(G,v)=0$, concluding the proof of the Theorem.
\end{proof}

\begin{cor}
  \label{cor:LongExactLattice}
  The short exact sequence of Theorem~\ref{thm:ShortExact} induces a
  long exact sequence
\[
\ldots \to \HFmComb _{i+1} (G_{+1}(v)) \to \HFmComb _i (G-v)\to 
\HFmComb _i(G) \to \HFmComb _i (G_{+1}(v))\to \
\HFmComb _{i-1}(G-v) \to \ldots
\]
on $\delta$-graded lattice homology.
\end{cor}
\begin{proof}
  The short exact sequence of Theorem~\ref{thm:ShortExact} induces a
  long exact sequence on the homologies, and it is easy to see that
  both $\Psi _v$ and $\Phi _v$ respects the grading of a generator
  $[K,E]$ given by the cardinality of $E$, hence the long exact
  sequence admits the form stated in the corollary.
\end{proof}

Theorem~\ref{thm:ShortExact} also gives a long exact sequence 
 \[
\ldots \to \HFaComb ^{[n]} _{i+1} (G_{+1}(v)) \to \HFaComb ^{[n]}_i (G-v)\to
\HFaComb ^{[n]}_i(G) \to \HFaComb ^{[n]}_i (G_{+1}(v))\to \
\HFaComb ^{[n]}_{i-1}(G-v) \to \ldots.
\]
This is gotten by tensoring the short exact sequence
from Equation~\eqref{eq:minusses} with $\Field[U]/U^n$, and then taking the
associated long exact sequence in homology.

\end{document}